\documentclass[11pt]{amsart}
\advance\topmargin by -0.1in \advance\textheight by +0.1in
\advance\oddsidemargin by -0in \advance\textwidth by +0in

\usepackage{amssymb,amsmath,amscd,amsthm}
\usepackage{cancel}
\usepackage{hyperref}
\usepackage{color}
\usepackage[all]{xy}

\newcommand{\df}{\displaystyle\frac}
\newcommand{\CP}{{\mathbb {CP}}}

\newcommand{\HP}{{\mathbb {HP}}}
\newcommand{\OP}{{\mathbb {OP}}}

\newcommand{\Q}{{\mathbb Q}}

\newcommand{\Z}{{\mathbb Z}}

\newtheorem{theorem}{Theorem}[section]

\newtheorem{lemma}[theorem]{Lemma}
\theoremstyle{definition}

\newtheorem{remark}[theorem]{Remark}
\newtheorem{example}[theorem]{Example}

\title{Rational Analogs of Projective Planes}
\author{Zhixu Su}
\address{Department of Mathematics\\
University of California, Irvine\\
California, 92697\\
USA}
\email{zhixus@uci.edu}
\begin{document}

\maketitle

\begin{abstract}
In this paper, we study the existence of high-dimensional, closed, smooth manifolds whose rational homotopy type resembles that of a projective plane. Applying rational surgery, the problem can be reduced to finding possible Pontryagin numbers satisfying the Hirzebruch signature formula and a set of congruence relations, which turns out to be equivalent to finding solutions to a system of Diophantine equations.
\end{abstract}

\maketitle

\section{Introduction}

There are four kinds of projective planes, the well-known real, complex, quaternionic and octonionic projective planes. There does not exist any higher dimensional closed manifold having the topological structure of a projective plane. More precisely, for $n>8$, there does not exist any simply-connected $2n$-dimensional closed manifold $M$ such that
$$H^*(M;\Z)=\left\{
                            \begin{array}{ll}
                              \Z  & \ast =0, n, 2n; \\
                               0  & \hbox{ otherwise}.
                            \end{array}
                          \right.$$ 
This fact is a consequence of the well-known Hopf Invariant One Theorem.
If there were such a manifold $M^{2n}$ for $n>8$, then there would have to exist a Morse function with minimal number of critical points, giving a CW complex $X=e^0 \cup e^{n} \cup_\phi e^{2n}$ that is homotopy equivalent to $M$. This would require the existence of a Hopf invariant 1 attaching map $\phi: S^{2n-1}\rightarrow S^{n}$. But the only such maps are homotopic to the Hopf fibrations $S^{2k-1} \rightarrow S^k$ for $k=1,2,4,8$. 

Ignoring torsion, we ask if any rational analogs of projective planes exist in higher dimension. This paper proves the following result.
 
\begin{theorem}\label{main} After dimension 4,8, and 16, which are the dimensions of $\CP^2$, $\HP^2$ and $\OP^2$ respectively, the next smallest dimension where a rational analog of projective plane exists is 32, i.e, there exist 32-dimensional, simply-connected, closed, smooth manifolds $M$ such that
\begin{equation*}H^*(M;\Q)=\left\{
                            \begin{array}{ll}
                              \Q  & \ast =0, 16, 32; \\
                               0  & \hbox{ otherwise},
                            \end{array}
                          \right.
\end{equation*}
and there are infinitely many homeomorphism types of such manifolds.
\end{theorem}

From the desired intersection form, it is immediate that such a manifold exists only in dimension $4k$. We will first show that for $k\neq1$, there is no such manifold in dimension $4k$ where $k$ is odd. Then, as we study the candidate dimensions, 24 also turns out to give a negative answer. In dimension 32, we can find infinitely many homeomorphism types of rational projective planes in terms of their Pontryagin numbers.

The main tool to prove the results is the rational surgery realization theorem, which was first introduced by Barge in \cite[Theorem 1]{bg} and by Sullivan in \cite{s1}; equivalent statements can be found in Taylor-Williams \cite{tw}. The theorem gives a constructive answer to the existence question by finding pairings of $4i$-dimensional cohomology classes and a choice of fundamental class that act like Pontryagin numbers. In section 2, we state the rational surgery realization theorem, which will be phrased in a form that is ready for application to our problem. To make the theorem more accessible, a variant of the proof will be given. In section 3, we will prove Theorem \ref{main}.

{\bf Acknowledgments}

I would like to thank my advisor James Davis for many helpful conversations, comments and encouragement. I thank Diarmuid Crowley for pointing out a certain useful reference to me, and John Rickert for verifying some computations using number theory. I also thank the referee for providing helpful comments.

\section{Rational surgery}

Given a rational homotopy type, a natural question is whether there exists a closed manifold realizing the rational homotopy data. Compared to its integral version, the existence question in rational setting has a more explicit solution. Philosophically, this is due to the much simpler rational homotopy groups of spheres. Initiated by Barge \cite{bg} and Sullivan \cite{s1}, rational surgery constructs closed manifolds, that are rational homotopy equivalent to a proposed $\Q$-local space $X^n$, which is a CW complex whose homotopy groups are $\Q$-vector spaces. To get any positive answer, it is clearly necessary to start with a local space $X$ that satisfies Poincar\'{e} duality in rational coefficients. The ingredients for constructing a realizing manifold include choices of cohomology classes in $H^{4i}(X; \Q)$, which play the role of Pontryagin classes, and correspondingly, a suitable choice of fundamental class in $H_n(X;\Q)\cong H_n(X;\Z)\cong\Q$.

\begin{theorem}[Barge \cite{bg}, Sullivan \cite{s1}]\label{bs} Let $X$ be an $n=4k$-dimensional simply-connected, $\mathbb{Q}$-local,
$\mathbb{Q}$-Poincar\'{e} complex, where $k\neq1$. There exists a simply-connected $4k$-dimensional, closed, smooth manifold $M$, and a
$\mathbb{Q}$-homotopy equivalence $f:
M\to X$ 
if and only if there exist cohomology classes $p_i \in H^{4i}(X;\mathbb{Q})$  for $i=1, \ldots, k$, and a fundamental class $\mu \in
H_{4k}(X;\mathbb{Q})\cong\Q$ such that  	

\textup{(i)} the pairing of the $k$th $L$-polynomial of $p_i$'s and $\mu$ is equal to the signature of $X$, i.e., $\langle L_k(p_1,\ldots, p_k), \mu\rangle=\sigma(X)$;

\textup{(ii)} the intersection form $\lambda: H^{2k}(X;\Q)\times
H^{2k}(X;\Q)\rightarrow \Q$ defined as
$\langle\cdot\cup\cdot,\mu\rangle$ is isomorphic to a direct sum of copies of $\langle1\rangle$'s and $\langle-1\rangle$'s; and

\textup{(iii)} the pairings $\langle p_I, \mu\rangle=\langle p_{i_1}\cdots p_{i_r}, \mu\rangle$ over all the partitions $I=(i_1,\ldots,i_r)$ of $k$ form a set of Pontryagin numbers of a genuine closed smooth manifold, i.e., there exists a $4k$-dimensional closed smooth manifold $N$
such that $$\langle p_I(\tau_N),[N]\rangle=\langle p_I, \mu\rangle$$ for all partitions $I$ of $k$.

If the choice of $p_i$'s and $\mu$ satisfies all the conditions above, surgery theory will construct a $\Q$-homotopy equivalence $f: M\to X$ such that $f_*[M]=\mu$ and $f^*(p_i)=p_i(\tau_M)$, where $p_i(\tau_M)$ is the $i$-th Pontryagin class of the tangent bundle of $M$. As a consequence, the Pontryagin numbers $p_I[M]=\langle p_I, \mu\rangle$ for all partitions $I$ of $k$.

\end{theorem}

\begin{remark}
For the dimensions $n\not \equiv 0 \pmod{4}$, the answer to the existence question in \ref{bs} is always yes. Any choice of cohomology classes $p_i$'s will construct a rational nonzero degree normal map $f:M\to X$ such that $f^*(p_i)=p_i(\tau_M)$. Since $L_n(\Q)=0$ in such dimensions, the surgery obstruction always vanishes, and therefore a rational homotopy equivalence can be obtained.
\end{remark}

\begin{proof}  We will claim that condition (iii) guarantees a degree 1 normal map from a candidate manifold $M$ to $X$ so that the fundamental class of $M$ is sent to the chosen class $\mu$. Conditions (i) and (ii) ensure the vanishing surgery obstruction.

Consider any choice of cohomology classes 
$$p: X \xrightarrow{(p_1,\ldots,p_k)} \Pi K(\mathbb{Q},4i)\simeq BSO_{(0)}.$$  For $m>>n$, let $\gamma^m$ denote the universal plane bundle  over $BSO(m)$, define the map 
$$\overline{p(\gamma^m)}:BSO(m)\xrightarrow{(\overline{p_1(\gamma^m)},\ldots,\overline{p_i(\gamma^m)},\ldots,)}
\prod K(\Q,4i),$$
 where the total class $\overline{p(\gamma^m)}=1+\overline{p_1(\gamma^m)}+\cdots+\overline{p_i(\gamma^m)}+\cdots\in \prod H^{4i}(BSO(m);\Q)$ is the unique class such that $p(\gamma^m)\overline{p(\gamma^m)}=1$.
Let $PB$ be the homotopy pull-back space of $p$ and
$\overline{p(\gamma^m)}$, and $\xi^m$ the pullback bundle
of $\gamma^m$ over $PB$. We have constructed the two right-hand columns of the following diagram. Note that $\overline{p(\gamma^m)}$ and the projection map $pr_1$ are localization maps by construction.
\[\xymatrix{
\nu_{M} \ar[d] \ar[r] & \xi \ar[d] \ar[r] & \gamma^m \ar[d] \\
M \ar[rd]_f \ar[r]^g & PB \ar[d]^{pr_1} \ar[r]^>>>>>>{pr_2} & BSO(m)
\ar[d]^{\overline{p(\gamma^m)}} \\
 & X \ar[r]^>>>>>>{p} & \Pi K(\mathbb{Q},4i) }\]
For any homotopy class $\alpha \in \pi_{n+m}(T\xi^m)$, the corresponding  map $g: S^{m+n}\longrightarrow T\xi^m$ yields a candidate manifold $M=\alpha^{-1}(PB)$ by Thom-Pontryagin construction. Moreover, $g|_{M}: M\rightarrow PB$ is covered by
a bundle map from the normal bundle of $M$ to $\xi$. Chasing through the diagram, one can check that the input classes $p_i$ are pulled back to the Pontryagin classes $p_i(\tau_{M})$ through the composition map $f:=pr_1\circ g: M\rightarrow X$. 

To construct a degree 1 normal map so that $f_*[M]=\mu $, we need a particular class $\alpha\in\pi_{n+m}(T\xi^m)$ that maps to $\mu$ under the composition of the Hurewicz map, the Thom isomorphism, and the projection ${pr_1}_*: H_n(PB;\Z)\rightarrow H_n(X;\Z)$, which is shown the following diagram.
\begin{small}
\[
\xymatrix@R=9mm @C=3.5mm{
\alpha\in\pi_{n+m}(T\xi^m) \ar[ddd]^{{Tpr_1}_*}  \ar[dr] \ar[rrr]^{{T_{pr_2}}_*} & & &  \pi_{n+m}(T\gamma^m)\ni \beta \ar[dl] \ar[ddd]^{T\overline{p(\gamma^m)}_*}\\
& H_n(PB) \ar[d]^{{pr_1}_*} \ar[r]^>>>>{{pr_2}_*}  & H_n(BSO(m)) \ar[d]^{\overline{p(\gamma^m)}_*} \\
& \mu\in H_n(X) \ar[r]^>>>{p_*}  &  H_n(BSO(m)_{(0)}) \\
c_X\in\pi_{n+m}(T\widetilde{\nu}_X) \ar[ur] ^{\cong}\ar[rrr]^{{T_p}_*}  & & &
\pi_{n+m}(T\gamma^m_{(0)})  \ar[ul]_{\cong} }\]
\end{small}

In the lower right-hand corner of the diagram, $T\gamma_{(0)}^m$ is the Thom space associated to the rational spherical fibration $S_{(0)}^{m-1}\rightarrow S\gamma_{(0)}^m\rightarrow BSO(m)_{(0)}$, which is the localization of the sphere bundle $S^{m-1}\rightarrow S\gamma^m\rightarrow BSO(m)$.  The Hurewicz-Thom map
$$\pi_{n+m}(T\gamma^m_{(0)})\rightarrow H_{n+m}(T\gamma^m_{(0)}; \Z)\rightarrow  H_n(BSO(m)_{(0)}; \Z)$$
is an isomorphism since both the Thom space and the base space are $\Q$-local, and the Hurewicz map is a rational isomorphism for $m>>n$  (Milnor-Stasheff \cite[Theorem 18.3]{ms}). In the lower left corner, the rational spherical fibration $\widetilde{\nu}_X=p^*(S\gamma_{(0)}^m)$ and the associated Thom space $T\widetilde{\nu}_X$ are $\Q$-local, and the Hurewicz-Thom map $\pi_{n+m}(T\widetilde{\nu}_X)\rightarrow H_n(X;\Z)$ is also an isomorphism. Thus, for any fundamental class $\mu$, there is a class $c_X\in\pi_{n+m}(T\widetilde{\nu}_X)$ mapping to $\mu$. Moreover, it can be shown that the outer square of Thom spaces is a homotopy cartesian square (see Taylor-Williams \cite[Lemma 6.1]{tw} or \cite[Lemma 3.2.3]{su} for more details). All these observations together imply that if there exists a class $\beta\in\pi_{n+m}(T\gamma^m)$ in the upper right corner mapping to $p_*\mu\in H_n(BSO(m)_{(0)})$, then $\beta$ and $c_X$ would guarantee the existence of a desired class $\alpha$ that maps to $\mu$. 

Note that the Hurewicz-Thom map in the upper right-hand corner can be viewed as $\nu: \pi_{n+m}(T\gamma^m)\cong \Omega_n^{SO}\rightarrow H_n(BSO;\mathbb{Q})$, where $\nu(M)={\nu_M}_*[M]$, and $\nu_M$ is the classifying map of the normal bundle of a manifold $M$. Thus there is a $\beta$ mapping to $p_*\mu$ if and only if ${\overline{p(\gamma^m)}}_*^{-1}(p_*\mu)$ lies in the image of such map $\nu$.

If the input classes $\{p_i\}$ and $\mu$ together satisfy condition (iii), i.e., there exists a closed smooth manifold $N$ such that $\langle p_I(\tau_N),[N]\rangle=\langle p_I,\mu\rangle$, chasing through the diagram, we have 
$$\langle p_I(\tau_N),[N]\rangle=\langle p_I,\mu\rangle=\langle p_I(\gamma^m),p_*\mu\rangle=\langle\overline{p(\gamma^m)}_I,{\overline{p(\gamma^m)}}_*^{-1}(p_*\mu)\rangle.$$
Since $p(\tau_N)p(\nu_N)=1$, and $\overline{p(\gamma^m)}p(\gamma^m)=1$, the identity above implies that $\langle p_I(\nu_N),[N]\rangle=\langle
p_I(\gamma^m),{\overline{p(\gamma^m)}}_*^{-1}(p_*\mu)\rangle$, which is equivalent to saying that ${\overline{p(\gamma^m)}}_*^{-1}(p_*\mu)$ is the image of a manifold $N$ under the
homomorphism $\nu: \Omega_n^{SO}\rightarrow H_n(BSO;\mathbb{Q})$. This implies that $\pi_{n+m}(T\gamma^m)$ possesses the desired class $\beta$ and thus ensures the existence of $\alpha$, which finishes the proof that condition (iii) guarantees that there exists a degree 1 normal map such that $f_*[M]=\mu$.

Now surgery can be applied to alter the normal map to a rational homotopy equivalence if and only if the map has a vanishing surgery obstruction, which lives in the $L$ group 
$$L_n(\Q)\cong \Z\oplus \bigoplus_{\infty}\Z_2 \oplus \bigoplus_{\infty}\Z_4$$
(see Milnor-Husmoller \cite{mh}). The obstruction vanishes in its $\Z$-summand if and only if the signature $\sigma(M)=\sigma(X)$, which is equivalent to condition (i), since
\begin{eqnarray*}\langle L_k(p_1,\ldots, p_k),\mu\rangle&=&\langle L_k(p_1,\ldots, p_k), f_*[M]\rangle\\&=&\langle L_k(f^*p_1,\ldots,
f^*p_k),[M]\rangle\\&=&\langle
L_k(p_1(\tau_M),\ldots,p_k(\tau_M)),[M]\rangle\\&=&\sigma(M).
\end{eqnarray*} 
Condition (ii) requires the rational intersection form of $X$ to be a direct sum of $\langle1\rangle$'s and $\langle-1\rangle$'s, which guarantees the obstruction vanishes in the $\Z_2$ and $\Z_4$ summands of $L_{n}(\Q)$. This finishes the outline of the proof of Theorem \ref{bs}.
\end{proof}

\begin{remark}
One can also ask about the existence of a closed topological or piecewise-linear manifold realizing the rational homotopy type of projective planes. The realization theorem \ref{bs} still works for the $PL$ or $TOP$ category by changing the word ``smooth'' in condition (iii) to PL or topological.  \\
\end{remark}

\section{Rational projective planes}

In this section, we study the dimensions of rational projective planes. Recall that we are seeking the smallest dimension $4k$  $(>16)$ for which a simply-connected, closed, smooth manifold $M$ exists with 
$$H^*(M;\Q)=\left\{
                            \begin{array}{ll}
                              \Q  & \ast =0, 2k, 4k; \\
                               0  & \hbox{ otherwise}.
                            \end{array}
                          \right.$$ 
Equivalently, we determine the dimensions of simply-connected closed smooth manifolds that are rational homotopy equivalent to a $4k$-dimensional $\Q$-local, $\mathbb{Q}$-Poincar\'{e} complex $X$ where
$$ H^*(X;\Q)\cong\left\{
                            \begin{array}{ll}
                              \Q  & \ast =0, 2k, 4k; \\
                               0  & \hbox{ otherwise}.
                            \end{array}
                          \right.$$

\subsection{The target $\Q$-local space} First we construct $X$ from a Postnikov tower of rational principal fibrations. Let $X\rightarrow K(\Q, 2k)$ be the principal fibration with fiber $K(\Q, 6k-1)$ and $k$-invariant $\iota_{2k}^3$:
\begin{small}
\[\xymatrix{
            & K(\Q,6k-1)\ar[d] \ar[r] & K(\Q,6k-1) \ar[d] \\
            & X        \ar[d] \ar[r] & \ast \ar[d]\\
            & K(\Q,2k) \ar[r]^>>>>>>{\iota_{2k}^3}& K(\Q,6k). }\]
\end{small}Computing the spectral sequence, it is easy to check that $X$ has the desired rational cohomology ring $H^*(X;\Q)\cong \Q[x]/\langle x^3\rangle$ with $|x|=2k$. Notice that the signature $\sigma(X)=\pm1$ by our construction.

\subsection{Existence of rational projective planes}
\numberwithin{equation}{subsection}
Since $H^*(X;\Q)\cong \Q[x]/\langle x^3\rangle$, the input classes $p_i\in H^{4i}(X;\Q)$ are zero for all $i$ except $p_{\frac{k}{2}}$ and $p_k$. Plugging the constructed local space $X$ into  realization theorem \ref{bs}, the existence question of rational projective planes can then be answered as follows:

\begin{theorem} \label{3.1}For $k>4$, let $X$ be a $4k$-dimensional simply-connected $\Q$-local, $\mathbb{Q}$-Poincar\'{e} complex such that $H^*(X;\Q)\cong \Q[x]/\langle x^3\rangle$. There exists a simply-connected $4k$-dimensional, closed, smooth manifold $M$ with a $\Q$-homotopy equivalence $f:M\to X$ if and only if there exists a choice of cohomology classes $$p_{\frac{k}{2}} \in H^{2k}(X;\mathbb{Q}),\ \ \mbox{and}\ \  p_{k} \in H^{4k}(X;\mathbb{Q})
$$
together with a nonzero fundamental class $\mu \in H_{4k}(X;\Z)\cong
\mathbb{Q}$ such that 

\textup{(i)} $\langle L_{k}(0,\cdots,0,p_{\frac{k}{2}},0,\cdots,0,p_{k}), \mu\rangle=\pm1$;

\textup{(ii)} the intersection form on $H^{2k}(X;\mathbb{Q})$ with respect to
$\mu$ is isomorphic to $\langle1\rangle$ or $\langle-1\rangle$; and

\textup{(iii)} there exists a $4k$-dimensional, closed, smooth manifold $N$ such that $$\langle p_I(\tau_N),[N]\rangle=\langle p_I, \mu\rangle$$ 
for all partitions $I$ of $k$. \\
\end{theorem}

\subsection{Signature formula}
\numberwithin{equation}{subsection}
In Theorem \ref{3.1}, the signature condition (i) says:
$$s_{\frac{k}{2}, \frac{k}{2}} \langle p_{\frac{k}{2}}^2, \mu\rangle+s_k \langle p_k, \mu\rangle=\pm 1,$$where $s_{\frac{k}{2}, \frac{k}{2}}$ denotes the coefficient of $p_{\frac{k}{2}}^2$, and $s_k$ denotes the coefficient of $p_k$ in the $k$th $L$-polynomial.

$s_k$ can be calculated by the formula\begin{equation}\label{3.21}s_k=\displaystyle\frac{2^{2k}(2^{2k-1}-1)|B_{2k}|}{(2k)!},\end{equation} (Milnor-Stasheff \cite{ms}, Problem19-B), where $B_{2k}$ is the $2k$-th Bernoulli number\footnote{Here $B_{2k}$ denotes the even Bernoulli sequence $B_2=\frac{1}{6}, B_4=-\frac{1}{30}, B_6=\frac{1}{42}\ldots$.}. 

As mentioned in Anderson \cite[Lemma 1.5]{an}, $s_{\frac{k}{2}, \frac{k}{2}}$ can be calculated as
\begin{equation}\label{3.22}\begin{split}s_{\frac{k}{2}, \frac{k}{2}}&=\frac{1}{2}(s_{\frac{k}{2}}^2-s_{k})\\&=\frac{1}{2}\left(\left(\df{2^{k}\ (2^{k-1}-1)\ |B_{k}|}{k!}\right)^2-\df{2^{2k}\ (2^{2k-1}-1)\ |B_{2k}|}{(2k)!}\right).
\end{split}\end{equation}

From condition (i) and (iii), we can narrow down the candidate dimensions to $4k$ with $k$ even.
\begin{lemma}\label{3.2} For $k\neq 1$, there does not exist any rational projective plane in dimension $4k$ when $k$ is odd.
\end{lemma}

\begin{proof}
When $k$ is odd, the input Pontryagin class $p_i$ is nonzero only when $i=k$. Then condition (i) requires:
$$\langle L_{k}(0,\cdots,0,p_{k}),\mu\rangle=s_k\langle p_k,\mu\rangle=\pm1.$$
On the other hand, condition (iii) requires $\langle p_k,\mu\rangle$ to be a Pontryagin number of a closed smooth manifold, which must be an integer. Let $\frac{numer(s_k)}{denom(s_k)}$ denote the irreducible form of $s_k$, then $s_k\langle p_k,\mu\rangle=\frac{numer(s_k)}{denom(s_k)}\langle p_k,\mu\rangle=\pm1$ requires that the numerator $numer(s_k)=1$. But we will show that for $k\neq 1$,  $numer(s_k)>1$. We write
 $$s_k=\frac{2^{2k}(2^{2k-1}-1)|B_{2k}|}{(2k)!}=\frac{2^{2k}(2^{2k-1}-1)\,|numer(B_{2k})|}{(2k)!\,|denom(B_{2k})|},$$where $\frac{numer(B_{2k})}{denom(B_{2k})}$ is the irreducible form of $B_{2k}$. It is a fact that $denom(B_{2k})$ is given by the product of all primes $p$ for which $p-1$ divides $2k$. Also, these denominators are square-free and divisible by 6 by Milnor-Stasheff \cite[Page 284]{ms}. Therefore the $2$-adic order $\nu_2(denom(B_{2k}))=1$, i.e, $numer(B_{2k})$ is odd.  Since $k$ is odd and $k\neq1$, the base-2 expansion $2k=\sum_ {i=1}^{m}2^{n_i}$ has $m>1$. Thus the $2$-adic order $\nu_2((2k)!)=2k-m<2k-1$, and so 
 $$\nu_2(s_k)=\nu_2(2^{2k})-\nu_2((2k)!)-\nu_2(denom(B_{2k}))=2k-(2k-m)-1>0,$$ which implies that $numer(s_k)$ is divisible by 2, hence is greater than 1.\\
\end{proof}

\subsection{Dimension 24}
\numberwithin{equation}{subsection}
  Lemma \ref{3.2} indicates that $n=24$ is the next candidate.  It turns out that the signature formula can never be satisfied in this dimension.

\begin{lemma}\label{4} There does not exist any rational projective plane in dimension 24. 
\end{lemma}
\begin{proof}
Condition (i) requires existence of cohomology classes $p_3\in
H^{12}(X; \Q)\cong\Q$, $p_6\in H^{24}(X; \Q)\cong\Q$ and a choice of 
fundamental class $\mu\in H_{24}(X;\Z)\cong\Q$ such that 
$$s_{3,3} \langle p_{3}^2, \mu\rangle+s_6 \langle p_6, \mu\rangle=\pm 1.$$

 Let $\alpha$ be any nonzero class in $H^{12}(X,\Q)\cong \Q$. One can write 
$$p_3=a\alpha, \ \ p_3^2=a^2\alpha^2$$
 and $$p_6=b\alpha^2$$
  for some nonzero rational number $a$ and $b$. Correspondingly, let $[X]\in H_{24}(X,\Z)\cong\Q$ be the fundamental class such that $\langle\alpha \cup
\alpha,[X]\rangle=1$. 

In order to have a rational intersection form isomorphic to a direct sum of $\langle1\rangle$'s and $\langle-1\rangle$'s, we need to choose a fundamental class 
$\mu$ such that $\mu=\pm r^2[X]$ for some nonzero rational number $r$. 

Condition (iii) requires the pairings $\langle p_{3}^2, \mu\rangle$ and $\langle p_6, \mu\rangle$ to be integers. We may let $x$ and $y$ be the integers such that $x^2=a^2r^2$, $y=b\ r^2$, and so 
$$\langle p_3^2,\mu\rangle=\pm x^2, \ \langle p_6,\mu \rangle=\pm y$$ 

Altogether, condition
(i),(ii) and the integrality part of condition (iii) require the existence of integers $x$ and $y$
such that
\begin{equation}\label{3.31}s_{3,3} x^2+s_6 y=\pm1,\end{equation}
where the coefficients can be computed using formulas \eqref{3.21} and \eqref{3.22} to be 
$$s_{3,3}=-\frac{40247}{638512875}, \ \ s_6=\frac{2828954}{638512875}$$

The Diophantine equation \eqref{3.31} is equivalent to the quadratic residue problem of finding an integer $x$ such that 
\begin{align}\label{3.32}
-40247x^2 &\equiv\pm638512875\pmod{2828954}\\
x^2 &\equiv\pm(-40247)^{-1}\cdot638512875\pmod{2828954} \nonumber\\
x^2 &\equiv\pm(-296623)\cdot638512875\pmod{2828954}\nonumber\\
x^2 &\equiv\pm118951\pmod{2828954}\nonumber
\end{align}

Consider the prime factorization $2828954=2\cdot 23\cdot 89\cdot 691$ and the following two cases.

$\bullet$ {\it Case} (i). $x^2 \equiv118951\pmod{2828954}
$. The Jacobi symbol with modulus the prime factor 691 can be calculated as: 
\begin{small} $$\left(\dfrac{118951}{691}\right)=\left(\dfrac{99}{691}\right)=-\left(\dfrac{691}{99}\right)=-\left(\dfrac{97}{99}\right)=-\left(\dfrac{99}{97}\right)=-\left(\dfrac{2}{97}\right)=-1,$$\end{small} 
which implies that $\left(\dfrac{118951}{2828954}\right)=-1$.\\

$\bullet$ {\it Case} (ii). $x^2 \equiv-118951\pmod{2828954}
$. The Jacobi symbol with modulus the prime factor 23 can be calculated as: 
\begin{small}$$\left(\dfrac{-118951}{23}\right)=\left(\dfrac{-1}{23}\right)\left(\dfrac{118951}{23}\right)=(-1)\left(\dfrac{18}{23}\right)=(-1)\left(\dfrac{2}{23}\right)\left(\dfrac{3^2}{23}\right)=(-1)(1)(1)=-1,$$\end{small}
which implies that $\left(\dfrac{-118951}{2828954}\right)=-1$.

Therefore the congruence \eqref{3.32} has no solution. Hence equation \eqref{3.31} turns out to have no solution.

\end{proof}

In order to continue the analysis on the higher candidate dimensions,  we need to give condition (iii) of the rational surgery realization theorem \ref{bs} an explicit interpretation.

\subsection{Congruence relations among Pontryagin numbers}  Condition (iii) requires the set of pairings $\langle p_I, \mu\rangle$ to be Pontryagin numbers of a genuine closed smooth manifold. These integers form a sublattice in $\Z^{p(n)}$ which can be classified by a set of congruence relations. The following Hattori-Stong Theorem says that the Riemann-Roch Theorem and the integrality of Pontryagin numbers completely determine all the relations among the Pontryagin numbers of closed smooth manifolds. 

Prior to restating the Hattori-Stong Theorem, we provide the definition of the $KO$-theoretic Pontryagin character $e_i(\gamma)$ of the universal bundle $\gamma$ over $BSO$.  The total Pontryagin class of the universal vector bundle $\gamma$ can be formally expressed as $p(\gamma)=\Pi(1+x_j^2)$ by the splitting principle. 
The class $e_i(\gamma)\in H^*(BSO;\Q)$ is the $i$-th elementary symmetric polynomial of the variables $e^{x_j}+e^{-x_j}-2$, i.e.,
$$e_i(\gamma)=\sigma_i(e^{x_1}+e^{-x_1}-2, \ e^{x_2}+e^{-x_2}-2, \ \cdots).$$

Note that each class $e_i(\gamma)$ can be written as a polynomial in the Pontryagin classes $p_i(\gamma)$'s. This is because $e_i$ can be expanded as a symmetric polynomials of the variables $x_j^2$'s, but any symmetric polynomial can be expressed in terms of the elementary symmetric polynomials in the variables, which in our case are exactly the Pontryagin classes $p_i(\gamma)$'s, since the total class $p(\gamma)=\Pi(1+x_j^2)$.

\begin{theorem} [{Smooth Hattori-Stong Theorem. {Stong \cite[Page 207, Theorem (c)]{st2}}, {Madsen-Milgram  \cite [Theorem 11.19, 11.20, 11.21]{mm}}}] \label{hs} For closed smooth manifolds, the stable tangent bundle $\tau_N:N\to BSO$ induces a homomorphism
$$\tau: \Omega^{SO}_*/\mathrm{tor}\rightarrow H_*(BSO;\Q).$$
The image of the homomorphism $\tau$ is a lattice consisting of exactly the elements
$x\in H_*(BSO;\Q)$ such that
\begin{equation}\label{3.4.1}\left\{ \begin{array}{ll}
    \langle \Z[e_1(\gamma), e_2(\gamma),\cdots ]\cdot L(p_i(\gamma)), \ x\rangle\in\Z[\frac{1}{2}],\\
    \\
    \langle \Z[p_1(\gamma),p_2(\gamma),\cdots], \ x\rangle\in\Z,
\end{array}
                       \right.
\end{equation}
where $L(p_i(\gamma))$ is the total $L$-polynomial of the Pontryagin classes $p_i(\gamma)$'s.
\end{theorem}

Applying the smooth Hattori-Stong Theorem in our problem, we get the following results.
 
\begin{lemma} Condition (iii) in the rational surgery realization theorem \ref{bs} is equivalent to the following statement.

Given a local space $X$, there exist cohomology classes $p_i \in H^{4i}(X;\mathbb{Q})$ and a fundamental class $\mu \in
H_{4k}(X;\mathbb{Q})\cong
\mathbb{Q}$ such that
\begin{equation}\label{3.4.2}\left\{ \begin{array}{ll}
    \langle \Z[e_1, e_2,\cdots ]\cdot L(p_i), \ \mu\rangle\in\Z[\frac{1}{2}],\\
    \\
    \langle \Z[p_1,p_2,\cdots], \ \mu\rangle\in\Z,
\end{array}\right.
\end{equation}
where each class $e_i\in H^*(X;\Q)$ can be expressed as a polynomial of $p_i$'s in the same way that $e_i(\gamma)$ is expressed in terms of  $p_i(\gamma)$'s in the Hatorri-Stong Theorem \ref{hs}. Here $L(p_i)$ denotes the total $L$-polynomial of the classes $p_i$'s.

\end{lemma}

\begin{proof}
In Theorem \ref{hs}, since each $e_i(\gamma)$ can be written as a polynomial in the Pontryagin classes $p_i(\gamma)$, both lines of the congruence relations in \eqref{3.4.1} are equivalent to a set of integrality conditions on the Pontryagin numbers $\langle p_I(\gamma), x\rangle$.

For any $4k$-dimensional closed smooth manifold $N\in \Omega^{SO}_{4k}$, let $x=\tau_*[N]$. Since $$\langle p_I(\tau_N),[N]\rangle=\langle p_I(\gamma),\tau_*[N]\rangle=\langle p_I(\gamma), x\rangle,$$the relations on $\langle p_I(\gamma), x\rangle$ in \eqref{3.4.1} simultaneously determine a set of integrality conditions on the Pontryagin numbers $\langle p_I(\tau_N),[N]\rangle$. Therefore \eqref{3.4.1} characterizes all the possible Pontryagin numbers of a closed smooth manifold.

Condition (iii) in the main theorem \ref{bs} requires that the numbers $\langle p_I, \mu\rangle$ equal the Pontryagin numbers $\langle p_I(\tau_N),[N]\rangle$ for a certain genuine $4k$-dim closed smooth manifold $N$. Hence the numbers $\langle p_I, \mu\rangle$ must satisfy the same set of congruence relations that the Pontryagin numbers of a closed smooth manifold should satisfy. These relations are then expressed as \eqref{3.4.2}. 
\end{proof}

In our case, since all the Pontryagin classes are zero except in dimension $2k$ and $4k$, we may express the $e_i$ classes solely in terms of $p_{\frac{k}{2}}$ and $p_k$. 

The following example in dimension 16 illustrates how such expressions can be calculated explicitly in high dimensions.

\begin{example} Suppose we want to find the explicit congruence relations in dimension 16. The first thing we need to do is to express the 16-dimensional summand of $\Z[e_1, e_2, \ldots ]\cdot L$ in terms of the Pontryagin classes $p_i$'s. Since $e_i$ consists of classes of dimension no less than $4i$, the 16-dimensional classes live in 
\begin{equation}\label{3.4.3}\left(\Z\oplus\Z e_1\oplus\Z e_1^2\oplus\Z e_2\oplus\Z e_1e_2 \oplus\Z e_3\oplus\Z e_2^2\oplus\Z e_1e_3 \oplus\Z e_4\right)\cdot L\end{equation}
As we assume that the classes $p_i=0$ for all $i$ except $p_2$ and $p_4$, the total $L$-class is 
$$L=1+s_2 p_2+s_{2, 2}p_2^2+s_4p_4=1+\displaystyle\frac{7}{45}p_2-\displaystyle\frac{19}{14175}p_2^2+\displaystyle\frac{381}{14175}p_4,$$
as each of the $e_i$ classes can be written as a linear combination of $p_2,p_2^2$ and $p_4$. 

Taking $e_2$ for example, we first expand $e^{x_j}+e^{-x_j}-2$ as a power series
$$e^{x_j}+e^{-x_j}-2=x^2+\displaystyle\frac{x^4}{12}+\displaystyle\frac{x^6}{360}+\displaystyle\frac{x^8}{20160}+O(x^9).$$
Then analyze the symmetric polynomial:  
 \begin{small} 
\begin{eqnarray*}e_2&=&\sigma_2(e^{x_1}+e^{-x_1}-2, \ e^{x_2}+e^{-x_2}-2,\  \cdots)\\&=&\displaystyle\sum_{j, k}(e^{x_j}+e^{-x_j}-2)(e^{x_k}+e^{-x_k}-2)
\\&=&\displaystyle\sum_{j, k}\left(x_j^2+\displaystyle\frac{x_j^4}{12}+\displaystyle\frac{x_j^6}{360}+\displaystyle\frac{x_j^8}{20160}+O(x_j^9)\right)\left( x_k^2+\displaystyle\frac{x_k^4}{12}+\displaystyle\frac{x_k^6}{360}+\displaystyle\frac{x_k^8}{20160}+O(x_k^9)\right)
\\&=&
 \displaystyle\sum_{j, k}\left(x_j^2 x_k^2+\displaystyle\frac{x_j^4 x_k^4}{144}+\displaystyle\frac{x_j^2 x_k^6}{360}+\displaystyle\frac{x_j^6 x_k^2 }{360}\right) +\mbox{terms of degree other than 8 and 16}
\\&=& p_2+\cancelto{0}{\displaystyle\frac{p_1^2p_2}{360}}+\displaystyle\frac{p_2^2}{720}-\cancelto{0}{\displaystyle\frac{p_1p_3}{60}}+\displaystyle\frac{p_4}{40}+\mbox{terms of degree other than 8 and 16}.
\end{eqnarray*}
\end{small}

The condition from the summand $\Z e_2\cdot L$ is then
\begin{eqnarray*}\langle\Z e_2\cdot L, \ \mu\rangle&=&\Z\langle(p_2+\displaystyle\frac{p_2^2}{720}+\displaystyle\frac{p_4}{40})\cdot (1+\displaystyle\frac{7}{45}p_2-\displaystyle\frac{19}{14175}p_2^2+\displaystyle\frac{381}{14175}p_4), \ \mu\rangle
\\&=&\Z\langle \displaystyle\frac{113p_2^2}{720}+\displaystyle\frac{p_4}{40},\  \mu \rangle\in\Z[\frac{1}{2}]
\end{eqnarray*}Since we also require that $\langle p_2^2, \ \mu\rangle, \langle p_4, \ \mu\rangle\in \Z$, the condition is equivalent to the congruence relation
$$ 113\langle p_2^2, \ \mu\rangle+18 \langle p_4, \ \mu\rangle\equiv0\pmod{45}$$To find the complete set of congruence relations in dimension 16, one applies the same process to each of the summands in \eqref{3.4.3}.

Alternatively, one may use the approach that will be mentioned in Remark \ref{basis} to express the explicit congruence relations \eqref{3.4.2} in terms of the Pontryagin classes. We will continue using the Hattori-Stong Theorem and the method discussed in the example above in dimension 32 in the following section.\\ 
\end{example}

\subsection{Dimension 32}
\numberwithin{equation}{subsection}
  In this dimension, we ask about the existence of a simply-connected, closed, smooth manifold that is rational homotopy equivalent to a $\Q$-local space $X$ where 
$$H^*(X;\Q)\cong\left\{
                            \begin{array}{ll}
                              \Q  & \ast =0, 16, 32; \\
                               0  & \hbox{ otherwise}.
                            \end{array}
                          \right.$$
                          
Applying the realization theorem \ref{bs}, we look for cohomology classes $p_4$ and $p_8$ in $H^*(X;\Q)$, together with a choice of fundamental class $\mu\in H_{32}(X;\Z)$, such that conditions (i), (ii) and (iii) are satisfied. We can convert the problem to solving a system of diophantine equations. 

\begin{theorem}There exist rational projective planes in dimension 
32. 
\end{theorem}
\begin{proof} The signature condition (i) says
\begin{equation}\label{3.51}s_{4,4}p_4^2+s_8 p_8=\pm1,\end{equation}
where the coefficients can be computed by the formulas \eqref{3.21} and \eqref{3.22} to be 
$$s_{4,4}=-\frac{444721}{162820783125}, \ \ s_8=\frac{118518239}{162820783125}.$$

Similar to the analysis on dimension 24, condition (ii) and the integrality of Pontryagin numbers ensure that we may let 
$$\langle p_4^2,\mu\rangle=\pm x^2, \ \langle p_8,\mu \rangle=\pm y$$ 
where $x$ and $y$ are integers. The signature condition requires the existence of integers $x$ and $y$ such that:
\begin{equation}\label{3.51b}-444721 x^2+118518239  y=\pm 162820783125\end{equation}

To get the congruence relations in condition (iii)
\begin{equation}\label{3.52}
    \langle \Z[e_1, e_2,\cdots ]\cdot L, \ \mu\rangle\in\Z[\frac{1}{2}],
\end{equation}
we expand each basis class of $ \Z[e_1, e_2,\cdots ]$ as a power series in $p_4$ and $p_8$, since we care only about the cohomology classes in dimension 32,  higher degree classes in the representations having been discarded. The $e_i$ classes are calculated as follows:
$$\left\{ \begin{array}{ll}
   e_1=-\frac{1}{5040}p_4+\frac{1}{2615348736000}p_4^2-\frac{1}{1307674368000}p_8\\
   e_2=\frac{1}{40}p_4+\frac{3119}{435891456000}p_4^2+\frac{5461}{217945728000}p_8, \ \ \   e_1e_1=\frac{1}{25401600}p_4^2 \\
   e_3=-\frac{1}{3}p_4 + \frac{19}{39916800}p_4^2-\frac{31}{2851200}p_8, \ \ \    e_1e_2=-\frac{1}{201600}p_4^2, \ \ \ e_1^3=0\\
   e_4=p_4 + \frac{1}{1209600}p_4^2 + \frac{457}{604800}p_8, \ \ \  e_1e_3=\frac{1}{15120}p_4^2, \ \ \   e_2e_2=\frac{1}{1600} p_4^2\\
   e_5=-\frac{43}{2520}p_8,\ \ \ e_1e_4=-\frac{1}{5040}p_4^2,   \ \ \ e_2e_3=-\frac{1}{120}p_4^2, \\
   e_6=\frac{29}{180}p_8, \ \ \  e_2e_4=\frac{1}{40}p_4^2, \ \ \ e_3e_3=\frac{1}{9}p_4^2,\ \ \ e_1e_5=0\\
   e_7=-\frac{2}{3}p_8, \ \ \ e_3e_4=-\frac{1}{3}p_4^2,\ \ \  e_2e_5=0,\ \ \ e_1e_6=0 \\
   e_8=p_8, \ \ \ e_4e_4=p_4^2, \ \ \ e_3e_5=0, \ \ \ e_2e_6=0, \ \ \ e_1e_7=0   
\end{array}
                       \right.
$$

Multiplying the nonzero basis class on $e_i$ with the total $L$ class
$$L=1+L_4+L_8=1+\frac{381}{14175}p_4-\frac{444721}{162820783125} p_4^2 + \frac{118518239}{162820783125} p_8,$$
we obtain a basis for $\Z[e_1, e_2,\cdots ]\cdot L$ consisting of linear combinations of $p_4^2$ and $p_8$ in dimension 32. 
\begin{equation*}\left\{ \begin{array}{ll}
                            1\cdot L= -\frac{444721}{162820783125}p_4^2+\frac{118518239}{162820783125}p_8\\
                            \\
                           e_1\cdot L= -\frac{1992521}{373621248000} p_4^2-\frac{1}{1307674368000}p_8\\
                           \\
                           e_2\cdot L=\frac{292903727}{435891456000} p_4^2 + \frac{5461}{217945728000}p_8,\ \ \ \ \ \ \ e_1e_1\cdot L= \frac{1}{25401600}p_4^2\\
                           \\
                           e_3\cdot L=-\frac{357613}{39916800}p_4^2-\frac{31}{2851200}p_8,\ \ \ \ \ \ \ e_1e_2\cdot L= -\frac{1}{201600}p_4^2\\
                           \\
                           e_4\cdot L=\frac{32513}{1209600}p_4^2 + \frac{457}{604800}p_8, \ \ \ e_1e_3\cdot L=\frac{1}{15120}p_4^2, \ \ \ e_2e_2\cdot L=\frac{1}
                           {1600}p_4^2\\
                           \\
                           e_5\cdot L=-\frac{43}{2520}p_8, \ \ \ e_1e_4\cdot L=-\frac{1}{5040}p_4^2, \ \ \ e_2e_3\cdot L=-\frac{1}{120}p_4^2\\
                           \\
                           e_6\cdot L= \frac{29}{180}p_8,\ \ \  e_2e_4\cdot L=\frac{1}{40}p_4^2, \ \ \ e_3e_3\cdot L=\frac{1}{9}p_4^2\\
                           \\
                           e_7\cdot L=-\frac{2}{3}p_8, \ \ \ \ \ \ e_3e_4\cdot L=-\frac{1}{3}p_4^2\\
                           e_8\cdot L= p_8, \ \ \ \ \ \ e_4e_4\cdot L= p_4^2
                                                      \end{array}
                        \right.
\end{equation*}
Thus the integrality condition \eqref{3.52} holds true if and only if each basis class satisfies the relation
\begin{equation}\label{3.53}\langle -, \ \mu\rangle\in\Z[\frac{1}{2}].
\end{equation}

We have set up integers $x$ and $y$ so that $\langle p_4^2,\mu
\rangle=\pm x^2$ and $\langle p_8,\mu \rangle=\pm y$. As we simplify the coefficients and throw away the redundant relations, \eqref{3.53} is equivalent to the following set of congruence relations on integers $x$ and $y$.
\begin{equation}\label{3.54}\left\{ \begin{array}{ll}
    162820783125\ |\ -444721x^2+118518239 y\\
    638512875\ |\ 13947647x^2+2y\\
    212837625\ |\ 292903727x^2+10922y\\
    155925\ |\ 357613x^2+434y\\
    4725\ |\ 32513x^2 + 914y\\
    99225\ |\ x^2\\
    315\ |\ y.
                            \end{array}
                        \right.
\end{equation}

The last six congruence relations in \eqref{3.54} is equivalent to 
\begin{equation}\label{3.55}\left\{ \begin{array}{ll}
x^2\equiv 0 \pmod{3^4\cdot 5^2\cdot 7^2}, \\
 y \equiv 312282614\, x^2 \pmod{ 638512875}                      
  \end{array}
                        \right.
\end{equation}
Let $A$ and $B$ be integers, we may write 
\begin{equation}\label{3.55b}\left\{ \begin{array}{ll}
x^2=(3^4\cdot 5^2\cdot 7^2)\,A^2, \\
 y =(312282614)(3^4\cdot 5^2\cdot 7^2)\,A^2+(638512875)B                     
  \end{array}
                        \right.
\end{equation}
Plug in to the signature equation \eqref{3.51b}, we have 
\begin{eqnarray*}
-444721(3^45^27^2\,A^2) +118518239 [(312282614)(3^45^27^2\,A^2)+638512875B]\\
=\pm 162820783125,
\end{eqnarray*}
which is simplified to
\begin{align}\label{3.54b}
5751543975315A^2 + 118518239 B&=\pm 255 \nonumber\\
A^2 &\equiv\pm(5751543975315)^{-1}(255) \pmod{118518239 }\nonumber\\
A^2 &\equiv\pm59181964\pmod{118518239}
\end{align}

The Jacobi symbol
$$\left(\dfrac{59181964}{118518239}\right)=1,$$
which is a necessary condition for 19744467 to be a quadratic residue $\pmod {118518239}$. Checking the Jacobi symbol on each of the prime factors of $118518239=7\cdot 31\cdot 151\cdot 3617$, we have:
$$\left(\dfrac{59181964}{7}\right)=\left(\dfrac{59181964}{31}\right)=\left(\dfrac{59181964}{151}\right)=\left(\dfrac{59181964}{3617}\right)=1.$$
This indicates that 59181964 is indeed a quadratic residue $\pmod {118518239}$. 
Therefore \eqref{3.54b} has a solution. So we have shown that the system of Diophantine equations \eqref{3.51b} and \eqref{3.54}, which is equivalent to condition (i), (ii) and (iii), has infinitely many integer solutions. For example, the solution with the smallest positive $x$ value is
$$x=493965360, \ y=915578185531275.$$
\end{proof}

\begin{remark}\label{novikov}Recall that in the rational surgery realization theorem \ref{bs}, as we construct a $\Q$-homotopy equivalence $f: M\rightarrow X$, Pontryagin numbers of the resulting manifold $M$ are realized by the input pairings $\langle p_I, \mu\rangle=\langle p_I(\tau_M), [M]\rangle$. Therefore distinct integer solutions $x$ and $y$ in dimension 32 correspond to distinct pairs of Pontryagin numbers, which are homeomorphism invariants. So we have shown that there are infinitely many homeomorphism types of closed smooth manifolds that are rational analogs of projective planes. This ends the proof of our main theorem \ref{main}.
\end{remark}

\begin{remark}\label{basis} There is another approach to computing the congruence relations among Pontryagin numbers of closed smooth manifolds. The torsion-free part of the oriented cobordism ring is a polynomial ring over $\Z$, generated by a set of closed smooth manifolds in dimension $4k$, $k\geq 2$, 
$$\Omega^{SO}_*/\mbox{tor}\cong \Z[M^4, M^8, \ldots]$$
where the generator $M^{4k}$ can be taken as any manifold satisfying the following characteristic number property by Stong \cite[Page 207]{st1}: 
\begin{equation*}s_k(p_1,\ldots,p_k)[M^{4k}]=\left\{
                            \begin{array}{ll}
                              \pm q  & \hbox{ if $2k+1$ is a power of the prime $q$;} \\
                              \pm 1  & \hbox{ if $2k+1$ is not a prime power}.
                            \end{array}
                          \right.
\end{equation*}
Pontryagin numbers are oriented cobordism invariants. If we can find a set of basis manifold of $\Omega^{SO}_{4k}/$tor and compute the Pontryagin numbers, the congruence relations are then computable from the integer sub-lattice. Since $s_{k}[\CP^{2k}]=2k+1$, in many of the $4k$ dimensions (when $2k+1=q$, with $q$ a prime), $\CP^{2k}$ qualifies as a generator. For example, in dimension 8,
$$\Omega^{SO}_8\cong\langle\CP^2\times\CP^2\rangle\oplus\langle\CP^4\rangle.$$
In particular, for any closed smooth $8$-dimensional manifold $N$, the Pontryagin number of $N$ can be written as a linear combination
\begin{equation*}\left\{ \begin{array}{ll}
    p_{11}[N]=k p_{1,1}[\CP^2\times\CP^2]+\ell p_{1,1}[\CP^4]=18k+25\ell\\
       
    p_2[N]=k p_{2}[\CP^2\times\CP^2]+\ell p_{2}[\CP^4]=9k+10\ell\\
\end{array}
                        \right.
\end{equation*}\\
with $k, \ell\in\Z$. Thus, the congruence relations among Pontryagin numbers of any 8-dimensional closed smooth manifold $N$ can be computed as
\begin{equation*}\left\{ \begin{array}{ll}
    5\ |\ p_{1,1}[N] - 2p_{2}[N] \\
   9\ |\ 2p_{1,1}[N] - 5 p_{2}[N]\\
    p_{1,1}[N]\in\Z\\
     p_{2}[N]\in\Z.
\end{array}
                        \right.
\end{equation*}

However, in dimensions such as $4k=16$ and $4k=28$ where $2k+1$ is not a prime, $\CP^{2k}$ does not satisfy the characteristic number property, thus fails to qualify as a generator. We have to construct a generating manifold from a disjoint union of $\CP^{2k}$ and certain complex hypersurfaces (see Milnor \cite[Page 250]{m}). For example, in dimension $4k=16$, we have 
$$s_4(p)[9\CP^8+\mathcal{H}_{3,6}]=-3$$
and in dimension $4k=28$
$$s_7(p)[-85\CP^{14}-16\mathcal{H}_{3,12}+2\mathcal{H}_{5,10}]=-1$$
where $\mathcal{H}_{m,n}$ is the hypersurface of degree $(1,1)$ in $\CP^m\times\CP^n$. Once we obtain the generating manifolds, we still need to compute all the Pontryagin numbers $p_I$ for a set of basis manifolds, which is  very tedious.

\end{remark}

\end{document}